\documentclass{amsart}  
\usepackage{amsmath}
\usepackage{amsfonts}
\usepackage{amssymb}
\usepackage{amsthm}
\usepackage{subfigure}
\usepackage{wrapfig}
\usepackage{enumerate}
\usepackage{bm}
\usepackage[automark]{scrpage2}
\usepackage{graphicx}
\usepackage{url}

\newtheorem{theo}{Theorem}

\newtheorem{lem}[theo]{Lemma}
\newtheorem{cor}[theo]{Corollary}

\newtheorem{rem}[theo]{Remark}
\theoremstyle{definition}
\newtheorem*{nota}{Notation}

\newtheorem*{defi}{Definition}
\newtheorem*{con}{Convention}

\newcommand{\onz}{\mathrm{O}_n(\mathbb{Z})}

\newcommand{\setPermMatn}{\mathcal{P}_n}

\newcommand{\zn}{\mathbb{Z}^n}
\newcommand{\rn}{\mathbb{R}^n}
\newcommand{\rnpl}{\mathbb{R}_{\geq0}^n}

\newcommand{\xstar}{x^{\ast}\hspace{-0.05cm}}

\newcommand{\xstarfix}{x^{\ast}_{\mathrm{fix}}}
\newcommand{\XI}{X_I}

\newcommand{\FixGR}{\mathrm{Fix}_{G}(\rn)}

\newcommand{\Sn}[1]{\mathrm{S}_{#1}}

\newcommand{\An}[1]{\mathrm{A}_{#1}}

\newcommand{\largeIndSet}{I}
\newcommand{\smallIndSet}{J}
\newcommand{\neigh}{\mathcal{N}}
\begin{document}
\title{Symmetries in Integer Programs}
\author{Katrin Herr}
\address{Institut f\"ur Mathematik, MA 6-2\\
TU Berlin\\
%Stra\ss{}e des 17. Juni 136\\
10623 Berlin\\
Germany}
\email{herr@math.tu-berlin.de}  
\author{Richard B\"odi}
\address{IBM Zurich Research Laboratory\\
%S\"aumerstrasse 4\\
CH-8803 R\"uschlikon\\
Switzerland}
\email{rbo@zurich.ibm.com} 

\date{\today}

\begin{abstract}
The notion of symmetry is defined in the context of Linear and Integer Programming. Symmetric integer programs are studied from a group theoretical viewpoint. We investigate the structure of integer solutions of integer programs and show that any integer program on $n$ variables having an alternating group $A_n$ as a group of symmetries can be solved in linear time in the number of variables.
\end{abstract}
\maketitle

\section{Introduction}
\label{intro}

This paper continues to investigate symmetries of linear and integer programs which we have started in~\cite{herrboedi}. For the sake of completeness, we will briefly summarize the definitions and results from our previous paper.

In practice, highly symmetric integer programs often turn out to be particularly hard to solve. The
problem is that branch-and-bound or branch-and-cut algorithms, which are commonly used to solve
integer programs, work efficiently only if the bulk of the branches of the search tree can be
pruned. Since symmetry in integer programs usually entails many equivalent solutions, the branches
belonging to these solutions cannot be pruned, which leads to a very poor performance of the
algorithm.\\

Only in the last few years first efforts were made to tackle this irritating problem. In 2002,
Margot presented an algorithm that cuts feasible integer points without changing the optimal value
of the problem, compare~\cite{margot1}. Improvements and generalizations of this basic idea can be
found in~\cite{margot2,margot3}. In~\cite{linderoth1,linderoth2}, Linderoth et al. concentrate on
improving branching methods for packing and covering integer problems by using information about
the symmetries of the integer programs. Another interesting approach to these kind of problems has
been developed by Kaibel and Pfetsch. In~\cite{kaibel1}, the authors introduce special polyhedra,
called orbitopes, which they use in~\cite{kaibel2} to remove redundant branches of the search tree.
Friedman's fundamental domains in~\cite{friedman} are also aimed at avoiding the evaluation of
redundant solutions. For selected integer programs like generalized bin-packing problems there
exists a completely different idea how to deal with symmetries, see e.g.~\cite{fekete}. Instead of
eliminating the effects of symmetry during the branch-and-bound process, the authors exclude
symmetry already in the formulation of the problem by choosing an
appropriate representation for feasible packings.\\

In this paper we will examine symmetries of integer programs in their natural environment, the field of group theory. \\

\section{Preliminaries}
The main object of our studies are linear or integer programs, LP or IP for short:
\begin{equation}\label{LP1_ohne_rnpl}
\begin{split}
&\mathrm{max}\enspace  c^tx\\
&\mathrm{s.t.}\hspace{0.35cm}  Ax\leq b,\enspace x\in\rn\enspace,
\end{split}
\end{equation}
where $A\in \mathbb{R}^{m\times n},\enspace b\in \mathbb{R}^m$ and $c\in \rn\setminus\{0\}$. We are
especially interested in points that are candidates for solutions of an LP.
\begin{defi} A point $x\in\rn$ is \emph{feasible} for an LP if $x$ satisfies all constraints of the LP.
The LP itself and any set of points is \emph{feasible} if it has at least one feasible point.
\end{defi}
Hence, the set of feasible points $X$ of~(\ref{LP1_ohne_rnpl}) is given by
\[X:=\{x\in\rn\,|\,Ax\leq b\}\enspace.\]
\begin{con}
We call $X$ the \emph{feasible region}, $c$ the \emph{utility vector} and $n$ the \emph{dimension}
of~$\Lambda$. The map~$x\mapsto c^tx$ is called the \emph{utility function}, and the value of the
utility function with respect to a specific~$x\in\rn$ is called the \emph{utility value} of~$x$.
\end{con}
We can interpret the feasible region of an LP in a geometric sense. The following definition is
adopted from~\cite{schrijver}, p.~87.
\begin{defi}
A \emph{polyhedron} $P\subseteq \rn$ is the intersection of finitely many affine half-spaces, i.e.,
\[P:=\{x\in\rn\,|\,Ax\leq b\}\enspace,\]
for a matrix $A\in \mathbb{R}^{m\times n}$ and a vector $b\in \mathbb{R}^m$.
\end{defi}
Note that every row of the system~$Ax\leq b$ defines an affine half-space. Obviously, the set~$X$
is a polyhedron. Since every affine half-space is convex, the intersection of affine half-spaces --
hence, any polyhedron -- is convex as well. Therefore, we can now state the convexity of~$X$.
\begin{rem}\label{X_convex}
The feasible region of an LP is convex.
\end{rem}
Whenever we consider linear programs, we are particularly interested in points with maximal utility
values that satisfy all the constraints.
\begin{defi}
A \emph{solution} of an LP is an element $\xstar\in\rn$ that is feasible and maximizes the utility function.
\end{defi}
If we additionally insist on integrality of the solution, we get a so-called integer program, IP
for short. According to the LP formulation in~(\ref{LP1_ohne_rnpl}), the appropriate formulation
for the related IP is given by
\begin{equation}\label{IP1_ohne_rnpl}
\begin{split}
&\mathrm{max}\enspace  c^tx\\
&\mathrm{s.t.}\hspace{0.35cm}  Ax\leq b,\enspace x\in \zn\enspace,
\end{split}
\end{equation}
where $A\in \mathbb{R}^{m\times n}, b\in \mathbb{R}^m$ and $c\in
\rn\setminus\{0\}$.\\
Analogously, the set of feasible points $\XI$ of (\ref{IP1_ohne_rnpl}) is given by
\[\XI:=\{x\in\rn\,|\,Ax\leq b,\,x\in \zn\}=X\cap \zn\enspace.\]

\section{Symmetries}\label{subsection_symmetries}
In~\cite{herrboedi} symmetries of linear and integer programs have been defined as elements of~$\onz$, the group of all orthogonal matrices with integral entries,  that leave invariant the inequality system and the utility vector of the problem. Taking into account the usual
linear and integer programming constraint~$x\in\rnpl$, which forces non-negativity of the
solutions, the set of possible symmetries shrinks from~$\onz$ to the group of permutation
matrices~$\setPermMatn\leq\onz$.\\

We can always think of symmetry groups of linear or integer programs as subgroups of~$\Sn{n}$ by
\begin{rem}\label{action_of_Sn_on_Rn}
A group~$G\leq\Sn{n}$ acts on the linear space~$\rn$ via the $G$-equivariant mapping
\[\beta:\{1,\dots,n\}\to B:\enspace i\mapsto e_i\enspace,\]
where~$B$ is the set of the standard basis vectors~$e_1,\dots,e_n$ of~$\rn$.
\end{rem}

As in~\cite{herrboedi} we formulate the definition of symmetries
of linear programs and the corresponding integer programs simultaneously. Consider an LP of the
form
\begin{align}
\begin{split}\label{LP1}
&\mathrm{max}\enspace  c^tx\\
&\mathrm{s.t.}\hspace{0.35cm}  Ax\leq b,\enspace x\in\rnpl\enspace,
\end{split}
\intertext{and the corresponding IP given by}
\begin{split}\label{IP1}
&\mathrm{max}\enspace  c^tx\\
&\mathrm{s.t.}\hspace{0.35cm}  Ax\leq b,\enspace x\in\rnpl,\enspace x\in \zn\enspace,
\end{split}
\end{align}
where $A\in \mathbb{R}^{m\times n}, b\in \mathbb{R}^m$ and $c\in \rn\setminus\{0\}$. 
\vspace{0.2cm}

Note that the
LP~(\ref{LP1}) and the IP~(\ref{IP1}) have the additional constraint~$x\in\rnpl$.
\begin{nota}
An LP of the form~(\ref{LP1}) is denoted by $\Lambda$.
\end{nota}
Apparently, applying a permutation to the matrix~$A$ according to Remark~\ref{action_of_Sn_on_Rn}
translates into permuting the columns of~$A$. Since the ordering of the inequalities does not
affect the object they describe, we need to allow for arbitrary row permutations of the matrix~$A$.
The following definition takes these thoughts into account.
\begin{defi}
A \emph{symmetry of a matrix}~$A\in\mathbb{R}^{m\times n}$ is an element~$g\in\Sn{n}$ such that
there exists a row permutation~$\sigma\in\Sn{m}$ with
\[P_{\sigma}AP_{g}=A\enspace,\]
where~$P_{\sigma}$ and $P_g$ are the permutation matrices corresponding to~$\sigma$ and~$g$. The
\emph{full symmetry group of a matrix}~$A\in\mathbb{R}^{m\times n}$ is given by
\[\{g\in \Sn{n}\,\big{|}\,\exists\,{\sigma\in\Sn{m}}:\; P_{\sigma}AP_{g}=A\}\enspace.\]
A \emph{symmetry of a linear inequality system} $Ax\leq b$, where $A\in \mathbb{R}^{m\times n}$,
and $b\in \mathbb{R}^m$, is a symmetry~$g\in\Sn{n}$ of the matrix~$A$ via a row
permutation~$\sigma\in\Sn{m}$ which satisfies~$b^{\sigma}=b$.\\
A \emph{symmetry} of an LP~$\Lambda$
or its corresponding IP is a symmetry of the linear inequality system~$Ax\leq b$ that leaves the
utility vector~$c$ invariant. The \emph{full symmetry group} of~$\Lambda$ and the corresponding IP
is given by
\[\{g\in \Sn{n}\,\big{|}\,c^{g}=c,\,\exists\,{\sigma\in \Sn{m}}:\; (b^{\sigma}=b\,\wedge\,P_{\sigma}AP_{g}=A)\}\enspace.\]
\end{defi}
This is a definition of symmetry as it can be found in literature as well, see e.g.~\cite{margot2}.\\

\section{Symmetries in Integer Programming}\label{chapterIPapproach}
Due to~\cite{herrboedi}, Corollary 19, we notice that in the LP case,
transitivity of the group action already implies a one-dimensional set of fixed points, giving rise
to a one-dimensional linear program, which is the best possible result we can obtain. In this
section, it will turn out that the assumption of transitivity is not strong enough in the IP case
to lead to satisfying results. Moreover, we will see that not only the decomposition into orbits
but also the detailed structure of the symmetry group influences the complexity of integer
programs. The algorithm we are going to develop in this chapter builds on our approach for the
linear case.

We start with the consideration of the integer program corresponding to the LP given by

\[c^tx=x_1+x_2\]
subject to
\begin{alignat*}{3}
x_1\enspace&&&\leq\enspace&2.5&\\
&&x_2\enspace&\leq\enspace&2.5&\\
x_1\enspace&+\enspace&x_2\enspace&\leq\enspace&3.7&\enspace,
\end{alignat*}

\begin{figure}[htp]
% Use the relevant command to insert your figure file.
% For example, with the graphicx package use
\centering
 \includegraphics{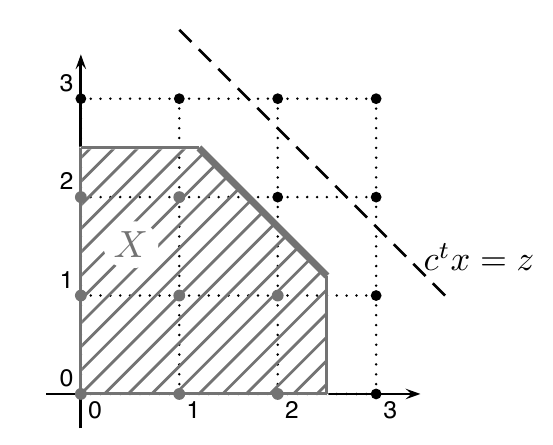}
% figure caption is below the figure
\caption{The set $X$ of feasible points}
\label{fig:1}       % Give a unique label
\end{figure}

Since we only have to handle two dimensions, we can solve this LP in a graphical
way. By pushing the blue line towards the feasible region $X$, we continuously decrease the utility
value $z$. The first non-empty intersection of the line $c^tx=z$ and $X$ then represents the set of
all solutions
of the LP, marked as a bold line.\\
Mathematically, the act of "pushing the dashed utility line" translates into looking at the affine
hyperplanes
\[H_{c,t}:=\ker(x\mapsto c^tx)+t\cdot c\]
for decreasing $t\in\mathbb{R}$. For every $x\in H_{c,t}$, there exists a vector \[x'\in \ker(y\mapsto c^ty)\] such that
\[x=x'+t\cdot c\enspace.\]
The computation of
\[c^tx=c^t(x'+t\cdot c)=\underbrace{c^tx'}_{=0}+t\cdot c^tc=t\|c\|^2\]
proves that all points of an affine hyperplane $H_{c,t}$ have the same utility value $t\|c\|^2$.
\begin{rem}\label{utility_val_const_on_hyperplane}
Given $t\in\mathbb{R}$ and a utility vector~$c$, the utility value is constant on the affine
hyperplane $H_{c,t}$.
\end{rem}
The family $(H_{c,t})_{t\in\mathbb{R}}$ consists of all affine hyperplanes that are orthogonal
to~$c$, thus they are parallel to each other. Therefore, every point is contained in exactly one
affine hyperplane $H_{c,t}$ for a specific utility vector~$c$.
\begin{lem}\label{x_in_specif_hyperplane_Hc_tx}
Given a point $x\in\rn$ and a vector $c\in\rn\setminus\{0\}$, the point $x$ is contained in the
affine hyperplane $H_{c,t_x}$ for $t_x=\frac{c^tx}{\|c\|^2}$.
\end{lem}
\begin{proof}
We define a vector $x'\in\rn$ by
\[x'=x-t_x \cdot c\enspace.\]
The computation of $c^tx'$ yields
\begin{align*}
c^tx'&=c^t(x-\frac{c^tx}{\|c\|^2}\cdot c)=c^tx(1-\frac{c^tc}{\|c\|^2})=0\enspace.
\end{align*}
That is, the vector $x'$ is an element of $\ker(y\mapsto c^ty)$. We conclude that the point $x=x'+t_x\cdot c$ is contained in
\[\ker(y\mapsto c^ty)+t_x\cdot c=H_{c,t_x}\enspace.\]
\end{proof}
Given a symmetry group~$G$, we know by Remark 12 of~\cite{herrboedi} that the line~$l$ through the origin
spanned by~$c$ is invariant under~$G$. Hence, this is also true for its orthogonal
complement~$\ker(x\mapsto c^tx)$. Since the line~$l$ is even pointwise fixed by~$G$, we finally
obtain the invariance of the affine hyperplanes~$H_{c,t}$ under~$G$. Referring to
Theorem 5 of~\cite{herrboedi}, we may add the constraint~$x\in H_{c,t}$ without losing
symmetry.
\begin{rem}\label{same_sym_intersection_X_H_ct}
Given an LP with utility vector~$c$ and a symmetry group~$G$, the intersection of the feasible
region and an affine hyperplane~$H_{c,t}$ is invariant under~$G$. In particular, the orbit~$x^G$ is
contained in the same affine hyperplane~$H_{c,t}$ as~$x$.
\end{rem}
Back to our example, we now want to find a solution of the corresponding IP. Typically, the set of
all solutions of the LP does not contain any integer solutions. Therefore, we have to push the line
further on to the closest integer point. In our example, this procedure leads to the following
situation.

\begin{figure}[htp]
% Use the relevant command to insert your figure file.
% For example, with the graphicx package use
\centering
 \includegraphics{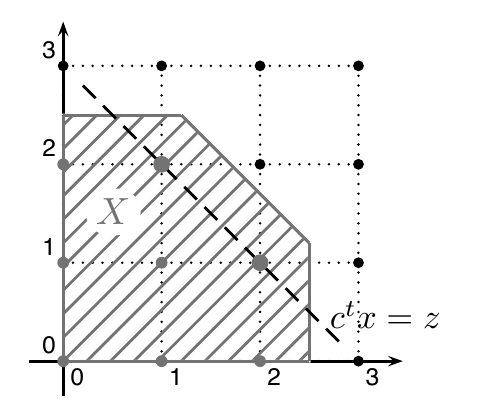}
% figure caption is below the figure
\caption{The two integer solutions}
\label{fig:2}       % Give a unique label
\end{figure}

Obviously, both accentuated points solve the IP. Furthermore, we observe that in this case,
hardly any of the affine hyperplanes $H_{c,t}$ contain integer points. Therefore, we introduce a
special term for affine hyperplanes that contain integer points.
\section{Integer-Layers}
\begin{defi}
A \emph{$c$-layer} is an affine hyperplane $H_{c,t}$ that contains at least one
integer point.
\end{defi}
The definition of $c$-layers immediately raises the following questions:
\begin{quote}
How many $c$-layers do we find, and what are the corresponding parameters $t$?
\end{quote}
To give a detailed answer to these questions, we need to distinguish two different types of utility
vectors $c$.
\begin{defi}
A utility vector is called \emph{projectively rational} if it is a real multiple of a rational
vector, and hence also of an integer vector. Otherwise, it is called \emph{projectively
irrational}. The \emph{coprime multiple} of a projectively rational utility vector $c$ is a real
multiple $c'\in\zn$ of $c$ whose entries $c'_1$ to $c'_n$ are coprime.
\end{defi}
We force uniqueness of the coprime multiple by demanding the first non-zero entry of $c'$ to be
positive. For example, the utility vector
\[(-2\sqrt{2},2\sqrt{2},4\sqrt{2},6\sqrt{2})^t=\sqrt{2}(-2,2,4,6)^t\]
is projectively rational with coprime multiple~$c'=(1,-1,-2,-3)^t$, whereas the vector
\[(\sqrt{2},\sqrt{6},2\sqrt{2},3\sqrt{2})^t=\sqrt{2}(1,\sqrt{3},2,3)^t\]
is projectively irrational.\\

Considering a utility vector $c$ and an arbitrary real multiple $c'\neq 0$ of $c$, we observe that
\[\ker(x\mapsto c^tx)=\ker(x\mapsto (c')^tx)\enspace.\]
Therefore the sets $(H_{c,t})_{t\in\rn}$ and $(H_{c',t})_{t\in\rn}$ of affine hyperplanes are
equal. In particular, this is also true for the corresponding layers.
\begin{rem}\label{c_rc_layers_equal}
Given a vector $c\in\rn\setminus\{0\}$, the set of $c$-layers is equal to the set of $(r\cdot
c)$-layers for every $r\in\mathbb{R}\setminus\{0\}$.
\end{rem}
We first want to study the configuration of $c$-layers for projectively rational utility
vectors~$c$. In this case, the $c$-layers are arranged in a very clear way.
\begin{theo}\label{c_layers_rational}
Given a projectively rational utility vector $c\neq 0$, the family of $c$-layers is given
by~${(H_{c',k\|c'\|^{-2}})}_{k\in\mathbb{Z}}$, where $c'$ is the coprime multiple of $c$.
\end{theo}
\begin{proof}
First, we prove that the family ${(H_{c',k\|c'\|^{-2}})}_{k\in\mathbb{Z}}$ contains every integer
point~$x\in\zn$.
Let~$x$ be an arbitrary integer point. By Lemma~\ref{x_in_specif_hyperplane_Hc_tx}, we already know
that~$x$ is contained in the affine hyperplane $H_{c',t_x}$, where
\[t_x=\frac{{c'}^tx}{\|c'\|^{2}}\enspace.\]
Since the value $k:={c'}^tx$ is integral for any $x\in\zn$, we get
\[x\in H_{c',t_x}\in{\left(H_{c',\frac{k}{\|c'\|^2}}\right)}_{k\in\mathbb{Z}}\enspace.\]
Next, we show that every affine hyperplane $H_{c',k\|c'\|^{-2}}$ actually is a $c'$-layer for every
$k\in\mathbb{Z}$ by specifying a certain integer point for each~$k$. Due to the coprimeness of the
entries of the vector $c'$, B\'ezout's Identity assures the existence of integral coefficients
$x_1,\dots,x_n$ such that
\[x_1\cdot c'_1+\dots+x_n\cdot c'_n=\gcd(c'_1,\dots,c'_n)=1\enspace.\]
Since the left side of the equation is equal to ${c'}^tx$, the integer point~$x$ is contained in
the affine hyperplane~$H_{c',\|c'\|^{-2}}$, see Lemma~\ref{x_in_specif_hyperplane_Hc_tx}. For the
same reason, the point $k\cdot x$ is an integer point in
\[H_{c',\frac{c'^t(k\cdot x)}{\|c'\|^2}}=H_{c',\frac{k\cdot (c'^tx)}{\|c'\|^2}}=H_{c',\frac{k}{\|c'\|^2}}\]
for every $k\in\mathbb{Z}$. Referring to Remark~\ref{c_rc_layers_equal}, we conclude that
$H_{c',k\|c'\|^{-2}}$ is a $c$-layer for every $k\in\mathbb{Z}$.
\end{proof}
If the entries of an integral utility vector $c$ are coprime, there are no integer points on the
line spanned by $c$ between the origin and the point~$c$. Since the standard lattice~$\zn$ is
invariant under translation by integer vectors, the number of $c$-layers and their arrangement are
the same between any two points~$mc$ and~$(m+1)c$ for~$m\in\mathbb{Z}$. Therefore, there are also
no integer points on the line spanned by~$c$ between any two points~$mc$ and~$(m+1)c$,
where~$m\in\mathbb{Z}$. Applying Theorem~\ref{c_layers_rational}, we can easily count the
$c$-layers between the hyperplane $H_{c,m}$ through the point~$mc$ and the affine hyperplane
$H_{c,m+1}$ through~$(m+1)c$.
\begin{cor}\label{c_layers_rational_between_two_points}
Given a projectively rational utility vector~$c$ with coprime multiple~$c'$, the number of
$c$-layers between~$mc'$ and~$(m+1)c'$ for any~$m\in\mathbb{Z}$ is equal to the squared euclidean
norm $\|c'\|^2$ of the coprime multiple~$c'$. These $c$-layers are the affine
hyperplanes~$H_{c',k\|c'\|^{-2}}$, where
\[k\in\left\{m\|c'\|^2,\dots,(m+1)\|c'\|^2-1\right\}\enspace.\]
\end{cor}
To put it precisely, the number of $c$-layers includes the layer through~$mc'$ but excludes the
layer through~$(m+1)c'$. The representation of the $c$-layers in
Corollary~\ref{c_layers_rational_between_two_points} allows us not only to count the layers but
also to access every single layer directly by its characteristic parameter~$k$.
\begin{con}
Given a projectively rational utility vector $c$ with coprime multiple~$c'$, the
$c$-layer~$H_{c',k\|c'\|^{-2}}$ is called the \emph{$k$-th~$c$-layer}.
\end{con}
Note that we always refer to the coprime multiple~$c'$ of a utility vector~$c$ when we talk about
$c$-layers. Figure~3 and Figure~4 give a graphical impression of the arrangement of $c$-layers
for two different utility vectors. In both figures, the outer two $c$-layers contain the two
integer points~$mc'$ and~$(m+1)c'$ on the line spanned by~$c$. In contrast to the situation in
Figure~3, we notice that in Figure~4, the two layers between the outer two layers do not cover
all integer points.

\begin{figure}[htp]
% Use the relevant command to insert your figure file.
% For example, with the graphicx package use
\centering
 \includegraphics[width=0.8\textwidth]{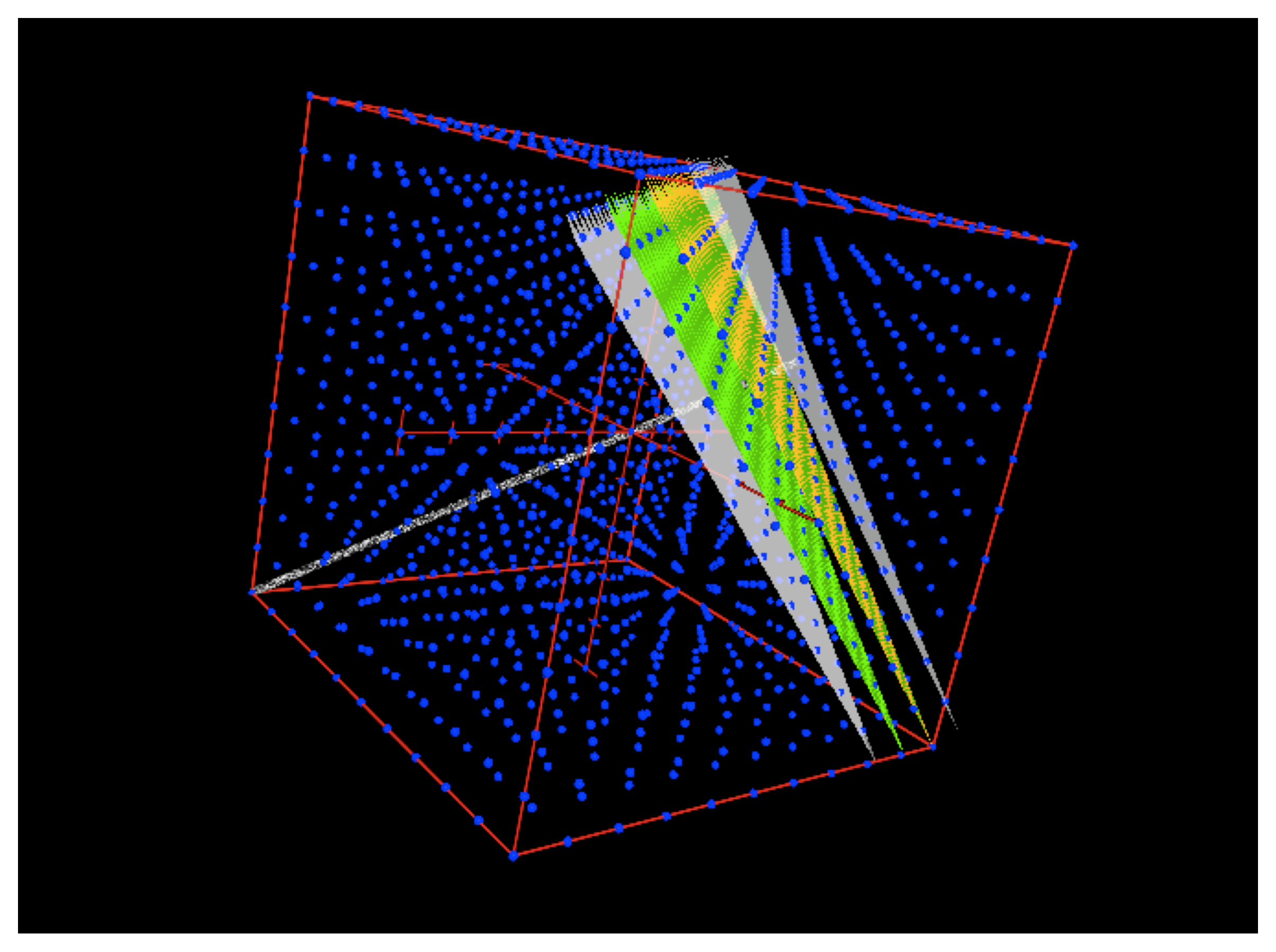}
% figure caption is below the figure
\caption{Some integer layers for $c' = (1,1,1)^t$}
\label{fig:3}       % Give a unique label
\end{figure}

\begin{figure}[htp]
% Use the relevant command to insert your figure file.
% For example, with the graphicx package use
\centering
 \includegraphics[width=0.8\textwidth]{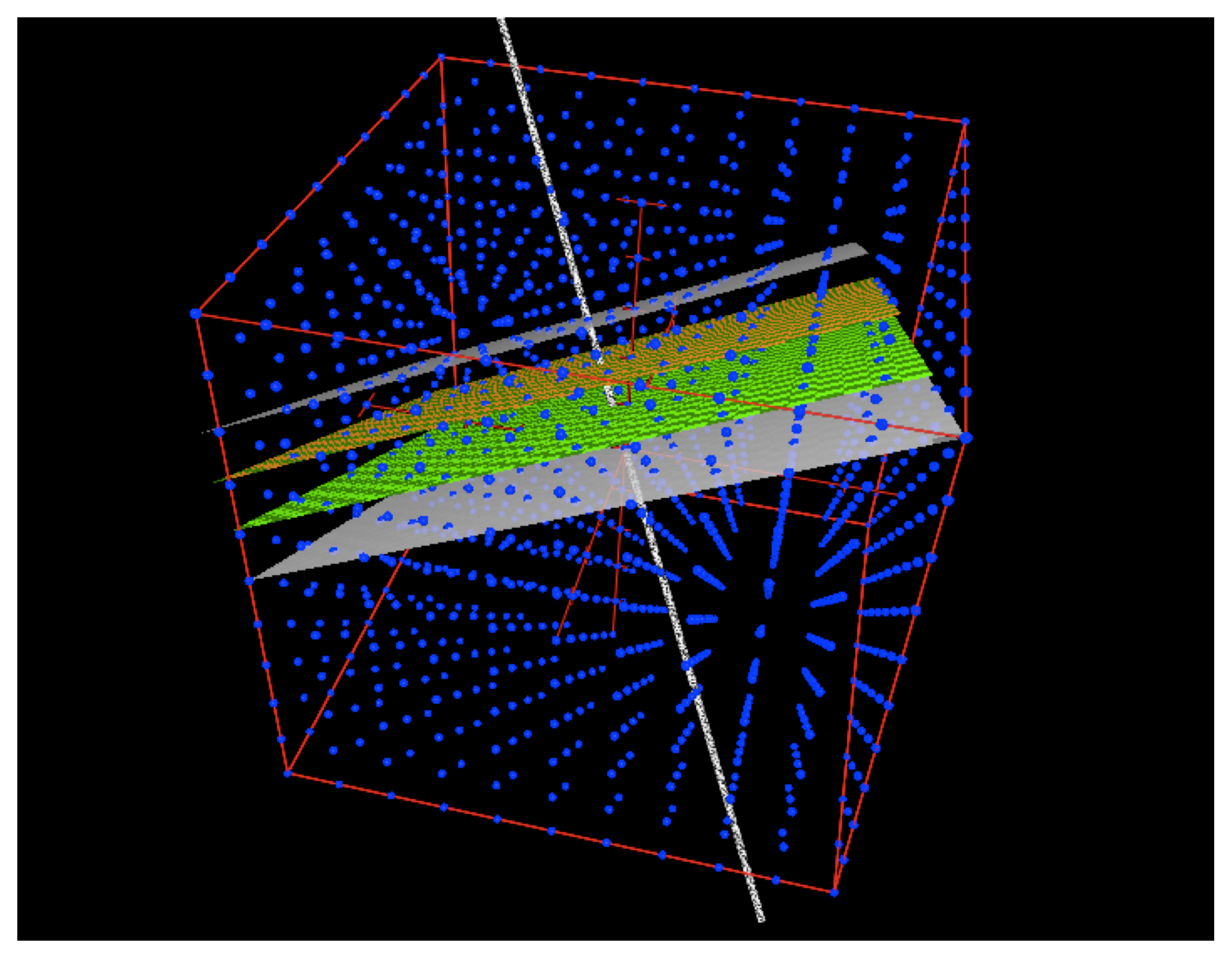}
% figure caption is below the figure
\caption{Some integer layers for $c' = (1,1,2)^t$}
\label{fig:4}       % Give a unique label
\end{figure}

Compared to the clear structure in the rational case, the arrangement of the $c$-layers for
projectively irrational utility vectors $c$ is rather complicated. In particular, we lose
finiteness of the number of $c$-layers between the origin and the point~$c$.
\begin{theo}
Given a projectively irrational utility vector~$c$, there exist infinitely many $c$-layers between
the origin and the point~$c$.
\end{theo}
\begin{proof}
Being projectively irrational, the vector $c$ is a real multiple of a vector $c'$ of the form
\[c'=(1,c'_2,\dots,c'_{j-1},r,c'_{j+1},\dots,c'_n)^t\enspace,\]
where $r\in\mathbb{R}\setminus\mathbb{Q}$, and $c'_i\in\mathbb{R}$ for $i\in\{2,\dots,j-1,j+1,\dots,n\}$. We define a sequence of integer points by
\[\left(x^{(k)}\right)_{k\in\mathbb{N}}:=\left((x_1^{(k)},\dots,x_n^{(k)})^t\right)_{k\in\mathbb{N}},\enspace x_i^{(k)}=\begin{cases}
-\lfloor rk\rfloor&\text{ if }i=1\\
k&\text{ if }i=j\\
0&\text{ otherwise }\enspace.
\end{cases}\]
Since $r$ is irrational, the equation
\[y_1+ry_2=y_1'+ry_2'\]
implies $y_1=y_1'$ and $y_2=y_2'$ for arbitrary integers $y_1,y_2,y_1',y_2'$. Hence, the utility values
\[c^tx^{(k)}=-\lfloor rk\rfloor+r\cdot k\]
are pairwise distinct for different $k\in\mathbb{N}$. Furthermore, we observe that
\[0\leq rk-\lfloor rk\rfloor\leq 1\leq \|c\|^2\enspace.\]
Therefore, every integer point of the sequence
$\left(x^{(k)}\right)_{k\in\mathbb{N}}$ is contained in a separate $c'$-layer
$H_{c',t_{x^{(k)}}}$, with $0\leq t_{x^{(k)}}\leq 1$, see
Lemma~\ref{x_in_specif_hyperplane_Hc_tx}. Referring to
Remark~\ref{c_rc_layers_equal}, this shows that we have infinitely many
$c$-layers between the origin and the point $c$.
\end{proof}
Apparently, stepwise sifting through $c$-layers is practicable only if the number of $c$-layers is
not too large but at least finite. Therefore, we will only follow up this method with respect to
projectively rational utility vectors.\\
The basic idea of our approach is to generalize the graphical method we studied in the beginning of
this chapter. By Remark~\ref{utility_val_const_on_hyperplane}, we know that the utility value is
constant on every $c$-layer. Hence, we start looking for a feasible integer point on the $c$-layer
next to an LP solution, which we can access directly because of the characterization given in
Corollary~\ref{c_layers_rational_between_two_points}. If there is no feasible integer point on this
layer, we go on to the $c$-layer with the next smaller utility value, again accessible due to
Corollary~\ref{c_layers_rational_between_two_points}.
The first feasible integer point we find by this method then is a solution to the IP problem.\\
Of course, it is not clear yet how to test feasibility of infinitely many integer points that are
contained in every single $c$-layer. Even if the IP has the additional assumption of positive
solutions, it is still not practicable to test a possibly exponentially large number of integer
points. Furthermore, we would like to detect infeasibility of an IP problem without exhaustive
testing of all possible integer points.\\
We will tackle these problems by means of symmetry. In contrast to the LP case, transitivity of the
group action in the IP case is not the end of the line but the initial assumption for our analysis.

\section{Transitive Actions}\label{subsection_transitive_actions}
Consider an LP with a symmetry group $G$ acting transitively on the standard basis. Then the
utility vector is a real multiple of the integer vector
\[\overline{c}:=(1,\dots,1)^t\enspace,\]
see Corollary 9 of~\cite{herrboedi}. Since the utility value is constant on any $c$-layer, compare
Remark~\ref{utility_val_const_on_hyperplane}, we obtain a useful characterization of the points on
the $k$-th $c$-layer by applying Lemma~\ref{x_in_specif_hyperplane_Hc_tx}.
\begin{rem}\label{sum_of_coord_equal_to_k}
Given the utility vector $c=(\gamma,\dots,\gamma)^t$, an integer point~$x$ is contained in the
$k$-th $c$-layer if and only if the sum of its coordinates is equal to~$k$.
\end{rem}
Further, we proved in Theorem 14 of~\cite{herrboedi} that the set of fixed points
$\FixGR$ is one-dimensional. Hence, it only consists of multiples of the utility vector, compare
Remark 12 of~\cite{herrboedi}. If we solve the LP according to the substitution algorithm we discussed in
the previous section, we get a solution of the LP of the form
\begin{align*}%\label{LP_solution_a_dots_a}
\xstarfix=(a,\dots,a)^t\in\FixGR\enspace.
\end{align*}
Therefore, the $c$-layer to start with in the transitive case is the $c$-layer next to this
solution given by~$H_{\overline{c},k\|\overline{c}\|^{-2}}$, where~$k=\lfloor na\rfloor$. If the IP
is feasible, we stop as soon as we find a feasible point. But what could be a reasonable stopping
criterion if the IP does not have any solutions? In general, the problem to decide whether an IP is
feasible or not, is NP-complete, see e.g.~\cite{schrijver}, p.~245. We are now going to study this
problem for transitive actions.%\\[0.7cm]
\subsection*{\textbf{Detecting Infeasibility}}
We start by defining a certain point of reference for every $c$-layer that shows an exceptional
property in the transitive case.
\begin{defi}
The \emph{center} of a $c$-layer is the intersection point of the $c$-layer and the line spanned by
the utility vector $c$.
\end{defi}
Note that in the transitive case, the center~$m_k$ of the $k$-th~$c$-layer is given~by
\[m_k=(\frac{k}{n},\dots,\frac{k}{n})^t\enspace.\]
If we consider two feasible points~$x_1$ and~$x_2$, the convexity of the feasible region~$X$
guarantees that the segment between~$x_1$ and~$x_2$ is feasible. Conversely, if only~$x_1$ is
feasible, then no point beyond~$x_2$ on the ray from~$x_1$ to~$x_2$ can be feasible. In particular,
we can apply this reasoning to the solution~$(a,\dots,a)^t$ of the LP and the center of any
$c$-layer. In the transitive case, the line through~$(a,\dots,a)^t$ and a center $m_k$ is equal to
the line generated by the utility vector~$c$. Hence, we get the following statement.
\begin{rem}\label{one_center_infeasible_all_lower_centers_infeasible}
Let~$(a,\dots,a)^t$ be a solution of an LP with a symmetry group acting transitively on the
standard basis. If the center of the $k$-th $c$-layer is infeasible for some $k\leq \lfloor
na\rfloor$, then the center of the $l$-th $c$-layer is infeasible for any $l\leq k$.
\end{rem}
Note that the following statement holds for any affine hyperplane~$H_{c,t}$,
where~$c=(\gamma,\dots,\gamma)^t$ and~$t\in\mathbb{R}$. However, we are interested in the result
only in relation to $c$-layers.
\begin{theo}\label{center_feasible}
Given an LP with a symmetry group $G$ acting transitively on the standard basis, a $c$-layer is
feasible if and only if its center is feasible.
\end{theo}
\begin{proof}
We cut down the feasible region of the LP to a feasible $c$-layer. The feasibility of the $c$-layer
assures that the substitution algorithm yields a solution to the resulting LP. By
Remark~\ref{same_sym_intersection_X_H_ct}, we know that $G$ still is a symmetry group of the
resulting LP. Hence, both LP problems share the same one-dimensional set of fixed points consisting
of the line $l$ spanned by $c=(\gamma,\dots,\gamma)^t$, compare Remark 12 of~\cite{herrboedi}.
Therefore, the solution to the resulting LP provided by the substitution algorithm is the
intersection point of $l$ and the $c$-layer, i.e., the center. Hence, in particular, the center is
feasible for the resulting LP, and therefore also for the original LP.
\end{proof}
Conversely, we conclude that there is no feasible point on a $c$-layer whose center is not
feasible. Hence, referring to Remark~\ref{one_center_infeasible_all_lower_centers_infeasible},
there are no feasible points in $c$-layers that have smaller utility values than the $c$-layer with
the first infeasible center. Therefore, we only need to search the layers beginning with the
$\lfloor na \rfloor$-th $c$-layer down to the $(n\lfloor a\rfloor)$-th $c$-layer. If the center of
one of the layers is infeasible, we already know that the IP is infeasible. Otherwise, we arrive at
the last layer and test the feasibility of the center~$(\lfloor a\rfloor,\dots,\lfloor
a\rfloor)^t$. Since the center is integral, we then either have found a solution or we conclude
that the IP is infeasible. Thus, the algorithm stops after having searched at most~$n$ $c$-layers,
see Corollary~\ref{c_layers_rational_between_two_points}.
\begin{cor}\label{testing_of_at_most_n_layers_suffices}
Let~$(a,\dots,a)^t$ be a solution of an LP with a symmetry group acting transitively on the
standard basis. Then stepwise sifting through the $\lfloor na \rfloor$-th $c$-layer down to the
$(n\lfloor a\rfloor)$-th $c$-layer either leads to a solution of the corresponding IP or reveals
its infeasibility. The algorithm stops after at most~$n$ steps.
\end{cor}
Corollary~\ref{testing_of_at_most_n_layers_suffices} discloses that the complexity of the
infeasibility problem only depends on the efficiency of the search algorithm that is used to sift
through a single $c$-layer. Therefore, we will now focus on the searching of a $c$-layer, i.e., on
the problem how to check the IP-feasibility of a $c$-layer without testing every single integer
point on the layer.
\subsection*{\textbf{Reducing to Neighbors}}
The main idea is to define an appropriate set of integer points -- the set of neighbors -- such
that the feasibility of any exterior integer point implies the feasibility of an integer point in
the same $c$-layer that belongs to the set of neighbors. In this case, it suffices to test the
feasibility of the neighbors. Unfortunately, transitivity is not strong enough to be able to reduce
the problem to the center which is not necessarily integral. In contrast to the LP case, compare
Corollary 19 of~\cite{herrboedi}, the Purkiss Principle of symmetric solutions of symmetric
problems is not suitable for IP problems. Therefore, the following definition leads to the smallest
possible set of neighbors with respect to the Euclidean distance.
\renewcommand{\labelitemi}{--}
\begin{defi}
Given a $c$-layer, a \emph{neighbor} is an integer point on the $c$-layer that has minimal
Euclidean distance from the center of the $c$-layer.
\end{defi}
Due to the simple structure of the utility vector in the transitive case (see ~\cite{herrboedi}), we can easily describe
the corresponding set of neighbors by using Remark~\ref{sum_of_coord_equal_to_k}.
\begin{rem}\label{set_of_neighbors}
Given the utility vector~$c=(\gamma,\dots,\gamma)^t$, and an integer $k=dn+r$,
where~$d\in\mathbb{Z}$ and $r\in\{0,\dots,n-1\}$, the set of neighbors $\neigh_k$ in the $k$-th
$c$-layer consists of all integer points that have $r$ entries equal to $d+1$ and $n-r$ entries
equal to $d$. The number of neighbors in the $k$-th $c$-layer is given by
\[|\neigh_k|=\binom{n}{r}\enspace.\]
\end{rem}
Consider two points~$x$ and~$y$ on the same hyperplane~$H_{c,t}$ satisfying
\[\|x\|^2\geq\|y\|^2\enspace.\]
Because of the orthogonality of~$H_{c,t}$ and the line spanned by~$c$, the Pythagorean theorem
yields that in this case, we also have
\[\|x-m\|^2\geq\|y-m\|^2\enspace,\]
where~$m$ is the center of~$H_{c,t}$. Therefore, we only need to determine the distance between the
origin and two different points on the same $c$-layer in order to decide which of them is closer to
the center.
\begin{rem}\label{distance_origin_suff_determ_distance_center}
Given two points~$x,y\in H_{c,t}$, then~$\|x\|^2\geq\|y\|^2$ implies
\[\|x-m\|^2\geq\|y-m\|^2\enspace,\]
where~$m$ is the center of~$H_{c,t}$. In particular, all elements of an orbit with respect to a
group~$G\leq\Sn{n}$ have the same distance to the center.
\end{rem}
The following lemma describes a method to approach the set of neighbors without leaving the
feasible region. In the proof, Remark~\ref{distance_origin_suff_determ_distance_center} helps us to
avoid technical difficulties.
\begin{lem}\label{convex_comb_closer_and_feasible}
Given an LP with symmetry group $G$ and a feasible point $x$ in the $k$-th $c$-layer, any interior
point of the segment determined by $x$ and a point $x^g\neq x$, where $g\in G$, is feasible and
closer to the center of the $k$-th $c$-layer than~$x$.
\end{lem}
\begin{proof}
Note that $x$, $x^g$ and -- since $c$-layers are affine hyperplanes -- any convex combination
\[y:=px+(1-p)x^g\]
are in the same $c$-layer, compare Remark~\ref{same_sym_intersection_X_H_ct}. Furthermore, the
feasibility of~$x$ implies the feasibility of~$x^g$, see Remark 6 of~\cite{herrboedi}, and
therefore also the feasibility of~$y$ due to the convexity of the feasible region~$X$. Hence,
referring to Remark~\ref{distance_origin_suff_determ_distance_center}, we only need to show that
the squared Euclidean norm of~$y$ is smaller than~$\|x\|^2$ for any~$p\in]0;1[$. Since~$\|x\|^2$ is
equal to~$\|x^g\|^2$, we can write
\begin{align*}
\|x\|^2&=p^2\|x\|^2+(1-p)^2\|x\|^2+2p(1-p)\|x\|^2=\\
&=p^2\|x\|^2+(1-p)^2\|x^g\|^2+p(1-p)\|x\|^2+p(1-p)\|x^g\|^2\enspace,
\end{align*}
and therefore
\begin{align*}
\|x\|^2-\|y\|^2&=\left(p^2\|x\|^2+(1-p)^2\|x^g\|^2+p(1-p)\|x\|^2+p(1-p)\|x^g\|^2\right)-\\
&-\left(p^2\|x\|^2+2p(1-p)\sum_{i=1}^n x_ix_{i^g}+(1-p)^2\|x^g\|^2\right)=\\
&=p(1-p)\|x\|^2+p(1-p)\|x^g\|^2-2p(1-p)\sum_{i=1}^n x_ix_{i^g}=\\
&=p(1-p)(\|x-x^g\|^2)>0,
\end{align*}
since $x\neq x^g$ and $p\in]0;1[$.
\end{proof}
Hence, we can approach the set of neighbors by considering convex combinations of two elements of
the same orbit. Since we are interested in solutions to IP problems, the convex combinations should
not only be feasible with respect to the LP but also with respect to the corresponding IP, i.e.,
they should be integral in addition. The next theorem shows that we can find such integral convex
combinations for any feasible integer point as long as the degree of transitivity of the symmetry
group is large enough.
\begin{theo}\label{always_feasible_integer_in_smaller_ring}
Let $G\leq \Sn{n}$ be a symmetry group of an LP acting
$(\left\lfloor\frac{n}{2}\right\rfloor+1)$-transitively on the standard basis, and~$n\geq 2$. If an
integer point~$x$ is feasible and not a neighbor, then there exists a feasible integer point in the
same $c$-layer that is closer to the center of the $c$-layer than~$x$.
\end{theo}
\begin{proof}
Due to the invariance of the standard lattice~$\zn$ under translation by integer vectors, we only
need to prove the statement for $c$-layers between the origin and the point~$(1,\dots,1)^t$. Let
$x$ be a feasible integer point on the $k$-th $c$-layer, where~$k\in\{0,\dots,n-1\}$. If all
coordinates of $x$ are equal, the point~$x$ is an element of the line spanned
by~$c=(\gamma,\dots,\gamma)^t$, and therefore the center of the $k$-th $c$-layer, thus a neighbor.
Otherwise, there exist at least two different coordinates~$x_i,x_{i'}$ of~$x$. We split the set of
indices into the two sets
\[\{i\,|\,x_i\equiv 0\text{ mod } 2\},\enspace\{i\,|\,x_i\equiv 1\text{ mod } 2\}\enspace.\]
Then one of the two sets -- denoted by~$\largeIndSet$ -- contains at
least~$\left\lfloor\frac{n+1}{2}\right\rfloor$ indices, while the other set~$\smallIndSet$ has at
most~$\left\lfloor\frac{n}{2}\right\rfloor$ elements. Therefore, we will use the
$\left\lfloor\frac{n}{2}\right\rfloor$-transitivity of~$G$ to control~$\smallIndSet$, and the
additional degree of transitivity to produce two different feasible integer points. We distinguish
the following two cases:
\begin{enumerate}[1)]
\item Suppose that~$x$ has two different coordinates~$x_i\neq x_{i'}$ of the same congruence class
modulo~$2$, that is, the corresponding indices~$i,i'$ are contained in the same set~$\largeIndSet$
or~$\smallIndSet$. Note that this condition is always satisfied if~$x$ has more than two different
coordinates. By the $(\left\lfloor\frac{n}{2}\right\rfloor+1)$-transitivity of~$G$, we then find a
permutation $g\in G$ such that
\[{i'}^g=i,\enspace \smallIndSet^g=\smallIndSet\enspace,\]
which implies~$\largeIndSet^g=\largeIndSet$. These assignments do not contradict each other since
we assumed that~$i'\in\smallIndSet$ if and only if~$i\in\smallIndSet$. Note that we do not require
non-emptiness of~$\smallIndSet$ in this case. By construction, all pairs of
coordinates~$(x_j,x_j^g)$ are in the same congruence class modulo~$2$ for all~$j=1,\dots,n$,
but~$x^g$ is different from~$x$ due to the assumption~$x_i\neq x_{i'}$. Hence, the convex
combination
\[y=\frac{1}{2}(x+x^g)\]
is an interior integer point of the segment determined by~$x$ and~$x^g$. Applying
Lemma~\ref{convex_comb_closer_and_feasible} to~$x$,~$x^g$ and~$y$, we conclude that~$y$ is a
feasible integer point that is closer to the center of the $k$-th $c$-layer than~$x$.
\item Otherwise, the point~$x$ has exactly two different coordinates~$x_i,x_{i'}$,
    where $i\in\largeIndSet$ and~$i'\in\smallIndSet$, that is, $x_k=x_i$ for all~$k\in\largeIndSet$,
    and $x_j=x_{i'}$ for all~$j\in\smallIndSet$. In this case, transitivity of~$G$ is sufficient to
    guarantee the existence of an appropriate permutation~$g\in G$ satisfying ${i'}^g=i$. Then~$x$
    and~$x^g$ are distinct elements of the orbit~$x^G$. Consider an interior point
    \[y=px+(1-p)x^g,\enspace p\in(0,1)\enspace,\]
    of the segment defined by~$x$ and~$x^g$. We want to determine a parameter~$p$ such that~$y$ is
    integral. Obviously, the coordinates~$y_l$ of~$y$ can only take the values
    \begin{alignat}{2}
    y_l&=px_{i\phantom{'}}+&(1-p)x_{i\phantom{'}}&=x_i\label{always_feasible_integer_in_smaller_ring_yl_one}\\
    y_l&=px_{i'}+&(1-p)x_{i'}&=x_{i'}\label{always_feasible_integer_in_smaller_ring_yl_two}\\
    y_l&=px_{i\phantom{'}}+&(1-p)x_{i'}&=p(x_i-x_{i'})+x_{i'}\label{always_feasible_integer_in_smaller_ring_yl_three}\\
    y_l&=px_{i'}+&(1-p)x_{i\phantom{'}}&=p(x_{i'}-x_i)+x_i\label{always_feasible_integer_in_smaller_ring_yl_four}\enspace.
    \end{alignat}
    Since~$x$ is integral, the coordinates of~$y$ of
    type~(\ref{always_feasible_integer_in_smaller_ring_yl_one}) and
    type~(\ref{always_feasible_integer_in_smaller_ring_yl_two}) are integral for any~$p\in(0,1)$. The
    coordinates of type~(\ref{always_feasible_integer_in_smaller_ring_yl_three})
    and~(\ref{always_feasible_integer_in_smaller_ring_yl_four}) are integral if~$p^{-1}$ divides the
    absolute value~$|x_i-x_{i'}|$. We may assume that~$x$ is not a neighbor. Since the
    sum of all coordinates is between~$0$ and~$n-1$, compare Remark~\ref{sum_of_coord_equal_to_k}, we
    therefore conclude that either the coordinates~$x_i$ and~$x_{i'}$ have different signs or one
    of the coordinates is equal to~$0$ and the other one is greater or equal than~$2$. In any case,
    we get~$|x_i-x_{i'}|\geq 2$. Hence, the choice
    \[p=\frac{1}{|x_i-x_{i'}|}\]
    is well-defined, and it guarantees that~$y$ is an interior integer point of the segment defined by~$x$ and~$x^g$.
    Again, we
    apply Lemma~\ref{convex_comb_closer_and_feasible} to~$x$,~$x^g$ and~$y$ in order to conclude
    that~$y$ is a
    feasible integer point that is closer to the center of the $k$-th $c$-layer than~$x$.
\end{enumerate}
\end{proof}
Iterated application of Theorem~\ref{always_feasible_integer_in_smaller_ring} demonstrates that the
set of neighbors is approachable over a sequence of feasible integer points starting from any
integer point within the feasible region. Therefore, we deduce the following statement.
\begin{cor}\label{feasibility_of_c_layer_implies_feasibility_of_set_of neighbors}
Let~$G\leq \Sn{n}$ be a symmetry group of an LP acting
$(\left\lfloor\frac{n}{2}\right\rfloor+1)$-transitively on the standard basis. Then the $k$-th
$c$-layer is feasible if and only if the set of neighbors~$\neigh_k$ is feasible.
\end{cor}
Thus, we may reduce the problem of looking for a feasible integer point on the whole $c$-layer to
just testing the set of neighbors, hence at most~$\binom{n}{r}$ points on that layer, given
$(\left\lfloor\frac{n}{2}\right\rfloor+1)$-transitivity of the symmetry group. In this respect, the
set of neighbors is representative for its layer. As a last step, we therefore concentrate on how
to test the set of neighbors in an efficient way.
\subsection*{\textbf{Testing Neighbors}}
Once more, we want to exploit our knowledge about the symmetries of an IP problem. To this end, we
consider the set of neighbors in the transitive case as described in Remark~\ref{set_of_neighbors}.
Obviously, this set is invariant under the action of any group~$H\leq\Sn{n}$. Hence, we can study
the action of a symmetry group~$G\leq\Sn{n}$ of an IP not only on the IP itself but also on the set
of neighbors in each layer. In particular, we are interested in the decomposition of the set of
neighbors into orbits. Since~$G$ leaves invariant the feasible region of the IP, the infeasibility
of one neighbor implies the infeasibility of any neighbor in the same orbit, compare
Remark 6 of~\cite{herrboedi}. Hence, we only need to test one neighbor in each orbit. Of
course, we would like to have a small number of orbits, preferably one orbit only, which is the
more likely the more symmetries we have. Due to the simple structure of neighbors, we can relax
assumptions on transitivity to assumptions on homogeneity, which is weaker in principle.
The following theorem provides an upper bound on the degree of homogeneity of the action on the
standard basis that suffices to guarantee transitivity of~$G$ on the set of neighbors.
\begin{theo}\label{transitive_action_on_set_of_neighbors}
Let $G\leq \Sn{n}$ be a symmetry group of an LP acting $k$-homo\-ge\-neous\-ly on the standard
basis, where~$k\in\{1,\dots,\left\lfloor\frac{n}{2}\right\rfloor\}$. Then the group~$G$ acts
transitively on the set of neighbors~$\neigh_{r+n\mathbb{Z}}$ for any integer
\[r\in\{0,\dots,k\}\cup\{n-k,\dots,n-1\}\enspace.\]
\end{theo}
\begin{proof}
Let~$r\in\{0,\dots,k\}$ and~$d\in\mathbb{Z}$. By Remark~\ref{set_of_neighbors}, we know that any
neighbor~$N\in\neigh_{r+nd}$ has exactly~$r$ coordinates of value~$d+1$, and~$n-r$ coordinates of
value~$d$. We denote the corresponding sets of indices by~$I_N^{d+1}$ and~$I_N^d$. Due to the
$k$-homogeneity of~$G$ on the standard basis, there exists a permutation~$g\in G$ that maps
the~$r\leq k$ elements of~$I_N^{d+1}$ of a neighbor~$N\in\neigh_{r+nd}$ to the~$r$ coordinates
in~$I_{N'}^{d+1}$ of any other neighbor~$N'\in\neigh_{r+nd}$. Since~$g$ is a bijection, the
remaining set of indices~$I_N^d$ is automatically mapped to~$I_{N'}^d$. Hence, we find a
permutation~$g\in G$ with~$N^g=N'$ for any two neighbors~$N,N'\in\neigh_{r+nd}$, i.e., the
group~$G$ acts transitively on the set of neighbors~$\neigh_{r+nd}$. For~$r\in\{n-k,\dots,n-1\}$,
we switch the roles of~$d$ and~$d+1$ and apply the same reasoning.
\end{proof}
Thus, the assumption of a symmetry group acting
$\left\lfloor\frac{n}{2}\right\rfloor$-homogeneously on the standard basis is sufficient to assure
transitivity on the set of neighbors in every layer. Combining Remark 6 of~\cite{herrboedi}
and Theorem~\ref{transitive_action_on_set_of_neighbors}, we conclude that in case of a
$(\left\lfloor\frac{n}{2}\right\rfloor+1)$-transitive action, we only need to test one neighbor in
any $c$-layer in order to decide its IP-feasibility.
\begin{cor}\label{testing_one_neighbor_suffices}
Let $G\leq \Sn{n}$ be a symmetry group of an LP acting
$(\left\lfloor\frac{n}{2}\right\rfloor+1)$-transitively on the standard basis, and~$n\geq 2$. Then
the set of neighbors~$\neigh_k$ is feasible if and only if any neighbor~$N\in\neigh_k$ is feasible.
\end{cor}
Now we are ready to bring together all the results of this section in order to deduce an applicable
algorithm.

\subsection*{\textbf{A Linear Algorithm for the Alternating and the Symmetric Group}}
If the number of dimensions~$n$ is greater or equal than~$5$, the assumption of
$(n-2)$-transitivity implies $(\left\lfloor\frac{n}{2}\right\rfloor+1)$-transitivity. Hence, we can
apply our results to any IP corresponding to an LP whose full symmetry group is isomorphic
to~$\An{n}$ or $\Sn{n}$, where~$n\geq 5$. For these problems,
Corollary~\ref{testing_of_at_most_n_layers_suffices} describes which $c$-layers need to be tested
and shows that we can stop after at most~$n$ layers.
Corollary~\ref{feasibility_of_c_layer_implies_feasibility_of_set_of neighbors} allows for reducing
the problem of testing the feasibility of every single point on a layer to simply testing the set
of neighbors. Finally, Corollary~\ref{testing_one_neighbor_suffices} guarantees that we only need
to check the feasibility of one neighbor per layer. Therefore, the following algorithm works
correctly, and it is linear in the number of dimensions~$n$.
\begin{cor}
Let~$n\geq 5$, and~$(a,\dots,a)^t$ be a solution of an LP with a symmetry group isomorphic
to~$\An{n}$ or $\Sn{n}$. Then testing the feasibility of one neighbor on every $c$-layer beginning
with the $\lfloor na \rfloor$-th $c$-layer down to the $(n\lfloor a\rfloor)$-th $c$-layer either
leads to a solution of the corresponding IP or reveals its infeasibility. The algorithm stops after
at most~$n$ steps.
\end{cor}

\section{Conclusion}
At the end of this chapter, we want to summarize the results and insights we gained. From an
application-oriented point of view, the most promising result seems to be the knowledge about the
existence and the configuration of $c$-layers, as it may contribute to a more systematic search for
integral solutions. If we consider also the decline in the utility value for descending $c$-layers,
the check for maximality becomes redundant, and the testing of the feasibility of integer points
comes to the fore.\\

Certainly, in practice, we cannot expect symmetry groups of integer programs that act
$(\left\lfloor\frac{n}{2}\right\rfloor+1)$-transitively on the standard basis, not even
transitively. Hence, the results of Section~\ref{subsection_transitive_actions} should be regarded
as an abstract approach to the question which role is played by symmetries in integer programs. On
the one hand, we proved that the complexity of integer programs with extremely large symmetry
groups like the alternating or the symmetric group is linear. On the other hand, we also notice
that the number of orbits of neighbors, thus the number of points to be tested, can get
exponentially large as soon as we consider integer problems with smaller symmetry groups. But in
any case, knowledge about symmetries helps us to reduce the number of points we need to check.
Therefore, symmetry in integer programs should not be demonized but seized in all its
potential.

\bibliographystyle{amsplain}      
\bibliography{SymIPArxiv}

\providecommand{\bysame}{\leavevmode\hbox to3em{\hrulefill}\thinspace}
\providecommand{\MR}{\relax\ifhmode\unskip\space\fi MR }
% \MRhref is called by the amsart/book/proc definition of \MR.
\providecommand{\MRhref}[2]{%
  \href{http://www.ams.org/mathscinet-getitem?mr=#1}{#2}
}
\providecommand{\href}[2]{#2}
\begin{thebibliography}{10}

\bibitem{fekete}
S{\'a}ndor~P. Fekete and J{\"o}rg Schepers, \emph{A combinatorial
  characterization of higher-dimensional orthogonal packing}, Mathematics of
  Operations Research \textbf{29} (2004), 353--368.

\bibitem{friedman}
Eric~J. Friedman, \emph{Fundamental domains for integer programs with
  symmetries}, Combinatorial Optimization and Applications, First International
  Conference, COCOA 2007, Proceedings, 2007, pp.~146--153.

\bibitem{herrboedi}
Katrin Herr and Richard B{\"o}di, \emph{Symmetries in linear and integer
  programs}, Available at \url{http://arxiv.org/pdf/0908.3329}, 2009.

\bibitem{kaibel2}
Volker Kaibel, Matthias Peinhardt, and Marc~E. Pfetsch, \emph{Orbitopal
  fixing}, Integer Programming and Combinatorial Optimization, 12th
  International Conference, IPCO 2007, Proceedings, 2007, pp.~74--88.

\bibitem{kaibel1}
Volker Kaibel and Marc Pfetsch, \emph{Packing and partitioning orbitopes},
  Math. Program. \textbf{114} (2008), no.~1, 1--36.

\bibitem{margot1}
Fran\c{c}ois Margot, \emph{Pruning by isomorphism in branch-and-cut}, Math.
  Program. \textbf{94} (2002), no.~1, 71--90.

\bibitem{margot2}
\bysame, \emph{Exploiting orbits in symmetric {ILP}}, Math. Program.
  \textbf{98} (2003), no.~1-3, 3--21.

\bibitem{margot3}
\bysame, \emph{Symmetric {ILP}: Coloring and small integers}, Discrete Optim.
  \textbf{4} (2007), no.~1, 40--62.

\bibitem{linderoth1}
James Ostrowski, Jeff Linderoth, Fabrizio Rossi, and Stefano Smriglio,
  \emph{Orbital branching}, Integer Programming and Combinatorial Optimization,
  12th International Conference, IPCO 2007, Proceedings, 2007, pp.~104--118.

\bibitem{linderoth2}
\bysame, \emph{Constraint orbital branching}, Integer Programming and
  Combinatorial Optimization, 13th International Conference, IPCO 2008,
  Proceedings, 2008, pp.~225--239.

\bibitem{schrijver}
Alexander Schrijver, \emph{{Theory of linear and integer programming. Repr.}},
  {Chichester: John Wiley \& Sons. {XI}, 471 p.}, 1998.

\end{thebibliography}

\end{document}